\documentclass[10pt,a4paper]{amsart}

\usepackage[all]{xy}
\usepackage{amsthm}
\usepackage{amssymb} 
\usepackage{amsmath,amscd}
\usepackage{mathrsfs}
\usepackage{color}
\usepackage{bm}
\usepackage{enumerate}

\usepackage{adjustbox}

\def\turn!{\textup{!`}}

\def\mrb{\mathrm{b}}

\def\mre{\mathrm{e}}

\def\kk{{\mathbf k}}

\def\ZZ{{\Bbb Z}}

\def\sfr{{\mathsf{r}}}

\def\sfD{{\mathsf{D}}}

\def\sfT{{\mathsf{T}}}

\def\tuD{{\textup{D}}}

\def\tuH{{\textup{H}}}

\def\tuT{{\textup{T}}}

\def\id{\operatorname{id}}

\def\mod{\operatorname{mod}}

\def\Hom{\operatorname{Hom}}

\newcommand{\RHom}{\operatorname{\Bbb{R}Hom}}

\newcommand{\lotimes}{\otimes^{\Bbb{L}}}

\def\RHom{\operatorname{\mathbb{R}Hom}}

\newtheorem{lemma}{Lemma}[section]
\newtheorem{proposition}[lemma]{Proposition}
\newtheorem{theorem}[lemma]{Theorem}
\newtheorem{corollary}[lemma]{Corollary}

\newtheorem{remark}[lemma]{Remark}

\newtheorem{definition}[lemma]{Definition}

\theoremstyle{remark}

\def\SL{\operatorname{SL}}

\def\YM{\Lambda}

\def\YYM{\widetilde{\Lambda}}

\def\PPi{\widetilde{\Pi}}

\title[The Hilbert series of the preprojective algebras]
{
The Hilbert series of the preprojective algebras 
}

\date{}

\author{Hiroyuki Minamoto}

\date{}



\address{Department of Mathematics Osaka Metropolitan University, Sakai City, Japan}
\email{minamoto@omu.ac.jp}

\begin{document}
\maketitle
%
%\begin{enumerate}[(1)]
%
%\item Write the introduction carefully.
%
%\item Introduce ${}^{v}\!\YM$ for general $v$ before the  sincere case.
%
%\item Question. In the case $Q$ is non-Dynkin not necessarily acyclic, we have $h_{\Pi(Q)} = h_{\PPi(Q)}$. 
%From this fact we can deduce that $\PPi(Q)$ is concentrated in degree $0$.
%
%\end{enumerate}
%

%Question
%\begin{enumerate}[(1)]
%%\item 
%% does  Nakayama permutation matrix $P$ and the adjacency matrix $C$ commute?
%
%\item path length of the element $u \in \HH(\PPi)$ (is it homogeneous with respect to path length)?
%
%
%\item $3$-koszulity and Hilbert series?
%
%Golod-Shafarevich inequality for higher koszul algebras?
%
%\item other applications??
%
%
%\end{enumerate}

\begin{abstract}
The aim of this short note is to prove the formula of the Hilbert series of the preprojective algebras in arbitrary characteristic 
by making effective use of the formulas of the Hilbert series of  differential graded (dg) algebras with Adams grading. 
We also compute the Hilbert series of the quiver Heisenberg algebras, a special class of central extensions of preprojective algebras.  
\end{abstract}

\section{Introduction}

In this note, $\kk$ denotes a field and a quiver  $Q$ is  assumed to be connected, finite and acyclic unless otherwise is mentioned. 

The  aim of this short note is to prove the formula of the Hilbert series of the preprojective algebras $\Pi(Q)$ (Theorem \ref{202005031400}, 
Theorem \ref{202303221917}) by making effective use of the formulas of the Hilbert series of  differential graded (dg) algebras with Adams grading.   

The formula were first obtained by Malkin, Ostrik and Vybornov \cite[p.516, Theorem]{MOV} in the characteristic zero case. 
They used theory of  module categories over representations of quantum $\SL(2)$. 
(We note that they dealt with the preprojective algebra $\Pi(Q)$ of $T$-type, which is not dealt in this paper.)

Later Etingof and Eu \cite[Theorem 3.4.2]{Etingof-Eu}  dealt with a non-Dynkin quiver $Q$ which is \emph{not necessarily acyclic}  
and proved that the same formula holds true in any characteristic. 
As a consequence they showed that the preprojective algebra $\Pi(Q)$ is Koszul where $Q$ is  non-Dynkin  and not necessarily acyclic.  
Their idea is to introduced the \emph{partial preprojective algebras} that allows them to reduce the problem of $Q$ to that of a smaller quiver 
$Q'$. 
However, as one of the point to which the problem for a general quiver is finally reduced, 
it is necessary for them to  compute the Hilbert series $h_{\Pi(Q)}$ for an extended-Dynkin quiver $Q$ before establishing general formula.
This was done in \cite[Proposition 3.2.1]{Etingof-Eu} by a rather complicated method. 

In this paper, 
we see that the formula  of the Hilbert series $h_{\Pi(Q)}$ for  a non-Dynkin quiver can be  easily deduced  from a well-know fact in representation theory of a quiver, if once we establish basic properties of the Hilbert series of differential graded algebras with Adams grading. 
By the same method, we also compute the Hilbert series of the quiver Heisenberg algebras (QHA), a special class of central extensions of preprojective algebras \cite{Etingof-Rains, QHA}.  
Finally, using a result of QHA and basic fact of representation theory of quiver, 
we establish the formula for the Hilbert series $h_{\Pi(Q)}$ for a Dynkin quiver $Q$ in arbitrary characteristic.

\section*{Acknowledgment}
The author thanks Martin Herschend for a suggestion that leads to  Lemma \ref{202402182021}.

This note stemmed from discussion  in the preparations of  Winter School ``Koszul Algebra and Koszul Duality'' at Osaka City University 2021,  
which the author co-organized with  Aaron Chan, Ryo Kanda and Hyohe Miyachi. 
The author thanks them for helpful discussions. 
This work was partly supported by Osaka City University Advanced Mathematical Institute : 
MEXT Joint Usage Research Center on Mathematics and Theoretical Physics JPMXP 0619217849

The  author was partially supported by JSPS KAKENHI Grant Number JP21K03210.

\section{Hilbert series of complexes and differential graded (dg) algebras  with Adams grading}\label{202005031353} 

When we deal with a differential graded (dg) algebra equipped with an extra grading, 
to distinguish it from  the cohomological grading, 
the extra grading is often called \emph{Adams grading}. 
In the paper we follow this terminology. 
Then  a usual graded algebra can be said as an algebra with Adams grading.

In this Section \ref{202005031353}, $r\geq 1$ is a positive integer 
and 
$A = \kk^{\times r} = \kk e_{1} \times \cdots \times \kk e_{r}$ 
where $e_{i}$ is the $i$-th canonical idempotent element:
\[ 
e_{i} = ( \stackrel{1}{0}, \cdots,  \stackrel{i-1}{0},  \stackrel{i}{1},  \stackrel{i+ 1}{0}, \cdots,  \stackrel{r}{0}).\]

We denote \emph{the dimension matrix} $(\dim e_{i} M e_{j})_{ij}$ of a finite dimensional $A$-$A$-bimodule $M$ by $[M]$. 
Then, it is clear that 
\[
 [A] = 1 \ \textup{(the unit matrix)},  \ [M^{*}] = [M]^{t} \ \textup{(transpose)}, \ [M\otimes_{A} N ]  = [M][N] 
\]
where $M^{*} = \Hom_{\kk}(M, \kk)$.

We regard $A$ as an Adams graded algebra concentrated in degree $0$. 
Recall that 
for an Adams  graded $A$-$A$-bimodule $M = \bigoplus_{ i} M_{i}$ which is locally finite dimensional (i.e., $\dim M_{i} < \infty$ for all $i\in\ZZ$) 
and bounded below (i.e.,  $ M_{i} = 0$ for $ i \ll 0$), 
its Hilbert series $h_{M}(t)$ is defined to be 
\[
h_{M}(t) := \sum_{i} [M_{i}]t^{i}. 
\]

To define the Hilbert series of complexes of $A$-$A$-bimodules with Adams grading, we need to introduce a boundedness condition. 

\begin{definition}\label{202303141831}
We call a complex  $M = ( \bigoplus_{n}M^{n}, d_{M})$ of graded $A$-bimodules $M^{n} = \bigoplus_{i } M_{i}^{n}$ 
\emph{suitably bounded} if it satisfies the following two properties:
\begin{enumerate}[(1)]

\item For each $i \in \ZZ$, we have $\dim \bigoplus_{n \in \ZZ} M_{i}^{n} < \infty$. 

In other words, for each $i$, the complex $M_{i}^{\bullet}$ is bounded and each term $M_{i}^{n}$ is finite dimensional. 

\item There exists $i_{0} \in \ZZ$ such that $M_{i}^{\bullet} = 0$ for all $i < i_{0} $.

\end{enumerate}
\end{definition}

For a suitably bounded complex $M$,  its \emph{Hilbert series} $h_{M}(t)$ is defined to be 
\[
h_{M}(t): = \sum_{n} ( -1)^{n} h_{M^{n}}(t) = \sum_{i,n} ( -1)^{n}[M_{i}^{n}]t^{i}. 
\]
\begin{remark}
It may be worth noting that $h_{M}(t)$ does not depend on the differential $d_{M}$ but 
only depends on the underlying cohomological graded $A$-$A$-bimodule $M^{\#}$ with Adams grading.  
\end{remark}

Let $n \in \ZZ$. We denote by $(n)$ (resp. $[n]$) the shift operation with respect to the 
Adams degree (resp. the cohomolgical degree), i.e., $(M[m](j))_{i}^{n} :=M_{i+j}^{n+m}$.

We collect basic properties which can be proved by standard arguments.

\begin{lemma}
Let $M$ and $N$ be suitably bounded complexes of graded $A$-bimodules.
The following assertions hold. 
\begin{enumerate}[(1)]
\item $h_{M[-n](-m)}(t) = ( -1)^{n}t^{m}h_{M}(t)$. 

\item $h_{M}(t) = h_{\tuH(M)}(t)$. 

\item If $M$ and $N$  are quasi-isomorphic, then $h_{M}(t) = h_{N}(t)$. 

\item $h_{M \otimes_{A} N} (t) = h_{M}(t) h_{N}(t)$. 
\end{enumerate}
\end{lemma}

From the last assertion of the lemma, we deduce 
\begin{corollary}\label{202005031416} 
Let $M$ be a suitably bounded complex of Adams graded $A$-$A$-bimodules 
and $T$ denotes the tensor algebra of $M$ 
\[
T = \tuT_{A} M =A \oplus M \oplus M\otimes_{A}M \oplus \cdots .
\]
Assume that  $M_{i}^{\bullet} = 0$ for all $i \leq 0$. 
Then $T$ is suitably bounded and  the Hilbert series $h_{T}(t)$ is 
\[
h_{T}(t) = \sum_{n \geq 0}h_{M}(t)^{n} = \frac{1}{1 -h_{M}(t)}.
\]
\end{corollary}

\section{The Hilbert series of the preprojective algebras} 

\subsection{The preprojective algebras}

Let $Q$ be a connected finite acyclic quiver with the vertexes $Q_{0} =\{ 1,2, \dots, r\}$.  
We set $A:= \kk Q_{0}= \kk e_{1} \times \cdots \times \kk e_{r}$ and equip the arrow space  
$V : = \kk Q_{1}$ with a canonical $A$-$A$-bimodule structure. 
If we regard the path algebra $\kk Q$ as an Adams graded algebra by path length, 
then it  is the tensor algebra over $A$ of the bimodule $V( -1)$
\[
\kk Q = \tuT_{A}(V(-1))  
\]
and hence  by Corollary \ref{202005031416} 
the Hilbert series $h_{\kk Q}(t)$ is 
\[
h_{\kk Q}(t) = \sum_{n \geq 0}([V]t)^{n} = \frac{1}{ 1 -[V] t}. 
\]

We denote by $\overline{Q}$ the double of $Q$. 
Namely, $\overline{Q}$ is obtained from $Q$ by formally adding 
an opposite arrow $\alpha^{*}: j \to i$ for each arrow $\alpha: i \to j$ of the original quiver $Q$. 
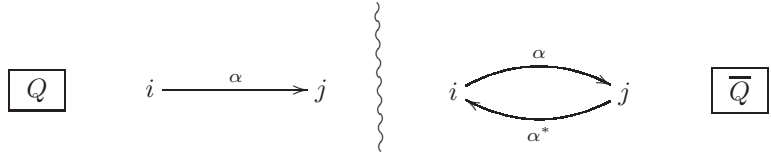
\[
\begin{xymatrix}{ & & \\ \boxed{ \ Q \ } &   i \ar[rr]^{\alpha} &&j }\end{xymatrix} \ \ \ \  \begin{xymatrix}@R=15pt@C=5pt{
\ar@{~}[ddd] \\  \\  \\ \\ 
}\end{xymatrix}  \ \ \ \ \ \ 
\begin{xymatrix}{ & \\ 
i \ar@/^10pt/[rr]^{\alpha} &&
 j  \ar@/^10pt/[ll]^{\alpha^{*}} &\boxed{ \ \overline{Q} \ } 
}\end{xymatrix} 
\]
Then the arrow space $\kk \overline{Q}_{1}$ can be identified with $V \oplus V^{*}$ where $V^{*} := \Hom_{\kk}(V, \kk)$.   
By the definition, the adjacency matrix $C$ of $\overline{Q}$ is given by 
\[C= [V]+ [V]^{t} = [V] +[V^{*}] = [\kk \overline{Q}_{1}]. 
\]

We recall the definition of  the preprojective algebra $\Pi(Q)$ of a quiver $Q$.  
For a vertex $i \in Q_{0}$, the \emph{mesh relation  $\rho_{i}$ at $i$} is the element of $\kk \overline{Q}$ 
that is given by 
\[ 
\rho_{i} := \sum_{\alpha \in Q_{1}: t(\alpha) = i} \alpha\alpha^{*} - \sum_{\alpha \in Q_{1}: h(\alpha) = i} \alpha^{*} \alpha. 
\]
The total sum  $\rho:= \sum_{ i \in Q_{0}} \rho_{i}$ is also  called  the \emph{mesh relation}.
The \emph{preprojective algebra} is defined to be the path of $\overline{Q}$ with mesh relations: 
\[
\Pi = \Pi(Q) =\frac{ \  \kk \overline{Q} \ }{ \   (\rho) \  } = \frac{ \kk \overline{Q} }{ (\rho_{i} \mid i \in Q_{0}) }.  
\]
We equip $\overline{Q}$ with an Adams grading  by path length: 
 \[ 
\textup{Ad}\deg \alpha := 1, \ \textup{Ad}\deg \alpha^{*} :=1 \textup{ for } \alpha \in Q_{1}. 
  \]
The mesh relations $\rho_{i}$ are homogeneous of Adams degree $2$ and consequently $\Pi(Q)$ is a graded algebra. 
Since the degree $n$-part of $\Pi(Q)_{n}$ of $\Pi(Q)$ is finite dimensional for all $n \geq 0$, we can take the Hilbert series 
\[
h_{\Pi(Q)}(t) := \sum_{n \geq 0} [\Pi(Q)_{n}] t^{n}.
\]

The formula for a non-Dynkin quiver is given in the following theorem. 

\begin{theorem}[{\cite{Etingof-Eu,MOV}}]\label{202005031400}
Assume that  $Q$  is a non-Dynkin quiver. 
Then the Hilbert series $h_{\Pi}(t)$ is given by the following formula: 
\[
h_{\Pi}(t) = \frac{1}{1 - Ct + t^{2} }. 
\]
\end{theorem}

Proof is given in Section \ref{202402211804}.

\subsubsection{The Dynkin case} 

Let $Q$ be a Dynkin quiver. To give a formula of the Hilbert series of $\Pi :=\Pi(Q)$, 
we need to introduce the Nakayama permutation matrix $P:= P_{\nu}$. 

First recall that a Nakayama automorphism  $\nu: \Pi \to \Pi$ is an algebra automorphism 
such that there is an isomorphism ${}_{\nu}\Pi \cong \tuD(\Pi)$ of $\Pi$-$\Pi$-bimodules. 
We use the explicit form of a Nakayama automorphism $\nu$ of $\Pi$ is given by Brenner-Butler-King  \cite[Definition 4.6]{BBK}. 
Then $\nu$ induces a permutation of the vertexes $Q_{0}$ which we denote by the same symbol $\nu$, 
i.e., for each $i \in Q_{0}$, there exists a unique vertex $\nu(i)$ such that $\Pi \nu(e_{i}) \cong \Pi e_{\nu(i)}$ as $\Pi$-modules. 
Then the Nakayama permutation matrix $P:= P_{\nu}$ is defined to be $P : = (\delta_{\nu(i)j})_{i,j\in Q_{0}}$ so that we have $[{}_{\nu}M] = P[M]$ for any $\Pi$-$\Pi$-bimodules $M$.
We note that the matrices $P$ and $C$ are commutative. 
Indeed, let $J$ denotes the Jacobson radical of $\Pi$. Then $C = [\kk \overline{Q}_{1}]= [J/J^{2}]$. 
It follows from Lemma \ref{202402182021} that $PC = CP$. 

Now we can give the formula for a Dynkin quiver $Q$. 

\begin{theorem}[{\cite{MOV} in the  characteristic zero case}]\label{202303221917}
Assume that  $Q$ is a Dynkin quiver. 
Then the Hilbert series $h_{\PPi}(t)$ is given by the following formula
\[
h_{\PPi}(t) = \frac{1+Pt^{h}}{1 - Ct + t^{2} }.
\]
\end{theorem}

Proof is given in Section \ref{202402211805}.

%
%\subsubsection{preliminaries}
%
%
%
%Generalizing this observation, we can easily deduce the following lemma. 
%
%\begin{lemma}\label{202005022042}
%Let $M_{1}, M_{2}, \cdots, M_{p}$ be (ungraded)  bimodules over $A$  and 
%$m_{1}, m_{2}, \cdots, m_{p}, n_{1}, n_{2} ,\cdots, n_{p} \in \ZZ$. 
%We take the tensor algebra 
%\[
%\Gamma := \tuT_{A} (M_{1}[-m_{1}]( -n_{1}) \oplus M_{2}[-m_{2}](-n_{2}) \oplus \cdots \oplus M_{p}[ -m_{p}](- n_{p}) ). 
%\]
%Then the Hilbert series $h_{\Gamma}(t)$ is 
%\[
%\begin{split}
%h_{\Gamma}(t) 
%&= \sum_{ n \geq 0} \left( (-1)^{m_{1}} [M_{1}] t^{n_{1}} + ( -1)^{m_{2}}[M_{2}]t^{n_{2}} + \cdots + (-1)^{m_{p}}[M_{p}]t^{n_{p}}\right)^{n} \\
%& = \frac{1}{ 1 - ( (-1)^{m_{1}} [M_{1}] t^{n_{1}} + ( -1)^{m_{2}}[M_{2}]t^{n_{2}} + \cdots + (-1)^{m_{p}}[M_{p}]t^{n_{p}} )}.
%\end{split}
%\]
%
%
%
%
%\end{lemma} 
%

\subsection{The derived preprojective algebras of a finite acyclic quiver $Q$}

We recall the definition of the derived preprojective algebra $\PPi :=\PPi(Q)$ of a quiver $Q$. 
The underlying cohomological graded algebra $\PPi^{\#}$  with Adams grading of $\PPi$ is given by the cohomological graded quiver 
$\widehat{Q}$ obtained from $\overline{Q}$ by formally adding a loop $s_{i}$ of $\textup{Ad}\deg s_{i} = 2, \textup{coh}\deg s_{i} = -1$ 
to each vertex $i \in Q_{0}$.
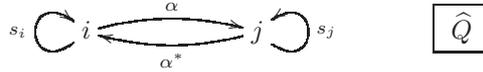
\[
\begin{xymatrix}{ 
i \ar@(ld, lu)^{s_{i}} \ar@/^5pt/[rr]^{\alpha} && j  \ar@(rd,ru)_{s_{j}} \ar@/^5pt/[ll]^{\alpha^{*}} && \ \boxed{ \ \widehat{Q} \ } 
}\end{xymatrix}
\]
In other words, 
$\PPi^{\#}$ is the tensor algebra over $A=\kk Q_{0}$ of the bimodule $ V(-1) \oplus V^{*}(-1) \oplus A[1](-2)$
\[
\PPi^{\#} = \tuT_{A}(V(-1) \oplus V^{*}(-1) \oplus A[1](-2))
\]
The differential $d$ of $\PPi$ is given by 
$
d (\alpha ) := 0, d(\alpha^{*} ) := 0, d(s_{i}) := - \rho_{i}.$ 
The values of $d$ for general homogeneous elements are determined from the Leibniz rule 
$d(xy) = d(x) y + ( -1)^{|x|}x d(y)$.

%
%
%
%
%Note that the bimodule $A[1]$ corresponds to the loops $s_{i}$ attached to  all vertexes. 
%Let $d_{\PPi}$ be the differential of $\PPi$. 
%Then by the definition 
%\[
%d_{\PPi}(s_{i} ) =  e_{i} \rho e_{i}.
%\] 
%
%We set the Adams grading of $\overline{Q}$ to be  the path length. 
%Then the mesh relation $\rho=\sum_{\alpha \in Q_{1}} \alpha \alpha^{*} - \alpha^{*} \alpha$ is of Adams degree $2$. 
%Therefore, to make the differential $d_{\PPi}$  preserve the Adams degree, 
%we have to set Adams degree of $s_{i}$ to be $2$.
%
%
%
%Thus 

Since the algebra $\PPi^{\#}$ is a tensor algebra over $A$, 
we can immediately obtain a formula for the Hilbert series $h_{\PPi}$.

\begin{proposition}\label{20230351859}
The Hilbert series $h_{\PPi}(t)$ of the derived preprojective algebra $\PPi$ is given by the following formula
\[
h_{\PPi}(t) = \frac{1}{1 - Ct + t^{2} }.
\]
\end{proposition}

\begin{proof}
First,  
we observe 
\[
h_{V(-1) \oplus V^{*}(-1) \oplus A[1](-2)} (t) = [V]t+[V^{*}]t - [A]t^{2} = Ct -t^{2}. 
\]
Applying  Corollary \ref{202005031416},  we complete the calculation  
\[
\begin{split} 
h_{\PPi}(t)  = h_{\PPi^{\#}}(t) =  \frac{1}{1- (Ct -t^{2})}  =\frac{1}{1- Ct +t^{2}}.  
\end{split}
\]
\end{proof}

\subsubsection{}

We recall basic properties of  the derived preprojective algebra. 
There is a canonical homomorphism $\PPi(Q) \to \Pi(Q)$ of dg-algebras with Adams grading, 
which induces an isomorphism $\tuH^{0}(\PPi(Q)) \xrightarrow{\cong } \tuH^{0}(\Pi(Q))=\Pi(Q)$ of Adams graded algebras.
The dg-algebra $\PPi(Q)$ is 
a model of Keller's $2$-Calabi-Yau completion of $\kk Q$ \cite{Keller: Calabi-Yau completion}, i.e., 
it is  quasi-isomorphic to the derived  tensor algebra of 
$\RHom_{\kk Q^{\mre}}(\kk Q, \kk Q^{\mre})[1] \cong 
\RHom_{\kk Q}((\kk Q)^{*}, \kk Q)[1]$ over $\kk Q$ (see for example \cite[Section 6.1]{QHA}). 
We  note that the endofunctor $\RHom_{\kk Q}((\kk Q)^{*}, \kk Q)[1]\lotimes_{\kk Q}- $ of the derived category $\sfD^{\mrb}(\kk Q\mod)$ 
is  the inverse $\nu^{-1}_{1}$ of $(-1)$-shifted Nakayama $\nu_{1}:= (\kk Q)^{*} [-1]\lotimes_{\kk Q} -$. 
\[
\PPi(Q) \simeq \bigoplus_{n \geq 0} (\RHom_{\kk Q}((\kk Q)^{*}, \kk Q)[1])^{\lotimes_{\kk Q} n}\simeq \bigoplus_{n\geq 0} \nu_{1}^{-n}(\kk Q)
\]

\subsection{Proof of Theorem \ref{202005031400}}\label{202402211804}

In the case where $Q$ is non-Dynkin,  
it is well-known that the complex $\nu_{1}^{-n}(\kk Q)$ is concentrated in $0$-th cohomological degree for all $n\geq 0$ and hence
 $\PPi(Q)$ is  quasi-isomorphic to $\Pi(Q)$ (see \cite[Example 2.8]{HIO}).  
Therefore, we conclude \[h_{\Pi}(t) = h_{\PPi}(t)=\frac{1}{1- Ct +t^{2}}. \] 
\qed

\section{The Hilbert series of the quiver Heisenberg  algebras}

\subsection{The quiver Heisenberg algebras}

Let $Q$ be a quiver. 
For an element $v =(v_{i})_{i \in Q_{0}} \in \kk Q_{0}$, we define the  central extension  ${}^{v}\!\YM :={}^{v}\!\YM(Q)$ of the preprojective algebra to be 
\[
{}^{v}\!\YM(Q) := \frac{\kk[z]\overline{Q}}{( \rho_{i} -v_{i}z e_{i} \mid i \in Q_{0} )}.
\]
This is a restricted version of the central extensions of preprojective algebras studied in \cite{Etingof-Rains}. 
Looking from another viewpoint  ${}^{v}\!\YM(Q)$ can be  regarded as a deformation family of a preprojective algebra. 
More general deformation families are  studied in \cite{Crawley-Boevey-Holland, QVB}. 
We set $\textup{Ad}\deg z :=2$ so that ${}^{v}\YM(Q)$ acquires  Adams grading.

In \cite{QHA}, the algebra ${}^{v}\!\YM$ is studied under the name \emph{quiver Heisenberg algebras} 
mainly in  the case where $v \in \kk Q_{0}$ is sincere (i.e., $v_{i} \neq 0$ for all $i \in Q_{0}$). 
We set ${}^{v}\varrho_{i} := v_{i}^{-1}\rho_{i}, \ {}^{v}\!\varrho :=\sum_{i \in Q_{0}}{}^{v}\!\varrho_{i}$ for $i\in Q_{0}$ and
${}^{v}\!\eta_{a} := a {}^{v}\!\varrho - {}^{v}\!\varrho a$ for $a \in \overline{Q}_{0}$. 
It is shown in \cite[Lemma 1.6]{QHA} that if $v\in \kk Q_{0}$ is sincere, then the algebra ${}^{v}\!\YM$ has an  isomorphisms
\[
{}^{v}\YM \xrightarrow{ \cong } \frac{  \kk \overline{Q}}{ \ ({}^{v}\!\eta_{a} | a \in \overline{Q}_{1}) \ }
\]
that preserves $\overline{Q}$ and sends $z$ to ${}^{v}\!\varrho$. 
It is clear that this isomorphism preserves  Adams grading.

To establish a formula of the Hilbert series of ${}^{v}\!\YM(Q)$, we introduce the derived quiver Heisenberg algebra ${}^{v}\YYM(Q)$.

\subsection{The derived quiver Heisenberg algebras}

Assume that $v\in \kk Q_{0}$ is sincere. 
We define the \emph{derived quiver Heisenberg algebra} ${}^{v}\!\YYM := {}^{v}\!\YYM(Q)$ in the following way. 
The underlying cohomological graded algebra ${}^{v}\!\YYM^{\#}$ is given by 
\begin{equation}\label{20191127}
\sfT_{A}(V(-1) \oplus V^{*}(-1) \oplus V^{*}[1](-3) \oplus V[1](-3) \oplus A[1](-4))  
\end{equation}
where $A = \kk Q_{0}$ and $V= \kk Q_{1}$. 
A part of corresponding cohomorogical graded quiver with Adams grading is given as below 
where we denote by $\alpha^{\circledast}$ the element of $V[1](-3)$ corresponding to $\alpha \in V$ and 
by $\alpha^{\circ}$ the element of $V^{*}[1](-3)$ corresponding to $\alpha^{*} \in V^{*}$. 
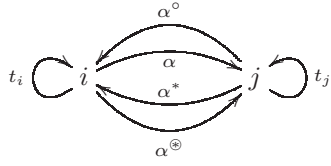
\[
\begin{xymatrix}{ 
i \ar@(ld, lu)^{t_{i}} \ar@/_20pt/[rr]_{\alpha^{\circledast}} \ar@/^10pt/[rr]_{\alpha} && j  \ar@(rd,ru)_{t_{j}} \ar@/^10pt/[ll]_{\alpha^{*}} 
\ar@/_20pt/[ll]_{\alpha^{\circ}}
}\end{xymatrix}
\]
%\centerline{
%\begin{tabular}{c|c|c|c|c|c|c}
%& $e_{i}$ & $\alpha$ & $\alpha^{*}$ &  $\alpha^{\circledast}$ &$\alpha^{\circ} $ & $t_{i}$ \\ \hline
%$\textup{ch}\deg $& $0$ & $0$  & $0$ & $-1$  & $-1$ & $-2$ \\ \hline
%$\textup{Ad}\deg$ & $0$ & $0$ & $0$ & $3$ & $3$ &  $4$ 
%\end{tabular}}

The differential $d$ of ${}^{v}\!\YYM$ is  defined in the following way. 
Then the underlying cohomological graded algebra \eqref{20191127}
 is freely generated by $\alpha, \alpha^{*}, \alpha^{\circ}, \alpha^{\circledast}, t_{i}$ 
for $\alpha \in Q_{1}$ and $i\in Q_{0}$. 
The values of $d$ for these generators are defined by the formulas:
\[
\begin{split}
&d(\alpha) := 0, d(\alpha^{*}) := 0,  d(\alpha^{\circ}) :  = - {}^{v}\!\eta_{\alpha^{*}},
 d(\alpha^{\circledast}) :  =  {}^{v}\!\eta_{\alpha}, \\
& d(t_{i})
= \sum_{\alpha: t(\alpha) = i} \alpha \alpha^{\circ}-  \sum_{\alpha: h(\alpha) = i} \alpha^{\circ} \alpha +
\sum_{\alpha: h(\alpha) = i} \alpha^{*} \alpha^{\circledast}-  \sum_{\alpha: t(\alpha) = i} \alpha^{\circledast} \alpha^{*}.
\end{split}
\]
The values of $d$ for general homogeneous elements are determined from the Leibniz rule 
$d(xy) = d(x) y + ( -1)^{|x|}x d(y)$.

In the same way of Proposition \ref{20230351859} we can compute the Hilbert series of the derived quiver Heisenberg algebra ${}^{v}\!\YYM$. 

\begin{proposition}\label{202303161404}
In the above setting,  we have 
\[
h_{{}^{v}\!\YYM}(t) = \frac{1}{1-Ct +Ct^{3}-t^{4}}= \frac{1}{(1-Ct+t^{2})(1-t^{2})}.
\]
\end{proposition}

\subsection{The non-Dynkin case} 

\begin{theorem}\label{202303161409}
Let $Q$ be a non-Dynkin quiver and $v \in \kk Q_{0}$. 
Then we have 
\[
h_{{}^{v}\!\YM}(t) = \frac{1}{(1-Ct+t^{2})(1-t^{2})}.
\]
\end{theorem}

\begin{proof}
First assume that $v \in \kk Q_{0}$ is sincere. 
Then  if $Q$ is non-Dynkin, then the derived quiver Heisenberg algebra ${}^{v}\!\YYM$ is quasi-isomorphic to 
the quiver Heisenberg algebra ${}^{v}\!\YM$ by \cite[Proposition 7.1]{QHA}. 
Thus we deduce the desired conclusion  from Proposition \ref{202303161404}. 

We deal with the general case. 
Recall that we label the vertexes as $Q_{0} = \{ 1,2, \dots, r\}$. 
We set $R:= \kk[x_{1}, \ldots, x_{r}]$ with $\textup{Ad}\deg x_{1} := 2, \dots, \textup{Ad}\deg x_{r} := 2$.
It is well-known that the deformation family 
\[
\Pi_{\bullet} :=\Pi(Q)_{\bullet} := \frac{R\overline{Q}}{ (\rho_{i} -x_{i}e_{i} \mid i \in Q_{0}) } 
\] 
introduced by Crawley-Boevey-Holland \cite{Crawley-Boevey-Holland} is flat over $R$ (see e.g. \cite[Theorem 7.3]{QHA}). 
Let $S:= \kk[y_{1},\dots, y_{r}, z]$ with $\textup{Ad}\deg y_{1} := 0, \dots, \textup{Ad}\deg y_{r} := 0, \textup{Ad}\deg z := 2$
 and $\phi: R \to S$ the $\kk$-algebra homomorphism such that $\phi(x_{i}) = y_{i}z$ for all $i\in Q_{0}$. 
We set ${}^{\bullet}\!\YM := S\otimes_{R} \Pi_{\bullet}$. 
Let $T:= \kk[y_{1}, \dots, y_{r}]$ and $\psi: T \hookrightarrow S$ the canonical inclusion. 
Since $S$ is flat over $T$, the algebra ${}^{\bullet}\!\YM$ is flat over $T$. 
Observe that for all $n \geq 0$, the Adams degree $n$-part ${}^{\bullet}\!\YM_{n}$ is a finite $T$-module and hence it is  a finite flat module over $T$.

For $v\in \kk Q_{0}$, we set  $\kappa(v) := T/(y_{1} -v_{1}, \dots, y_{r} -v_{r}) \cong \kk$, the corresponding residue field. 
Then the Adams graded algebra $\kappa(v) \otimes_{T} {}^{\bullet}\!\YM$ is isomorphic to ${}^{v}\!\YM$. 
Hence we have an isomorphism    $\kappa(v) \otimes_{T} {}^{\bullet}\!\YM_{n} \cong {}^{v}\!\YM_{n}$ of  $\kk$-vector spaces for each $n \geq 0$.  
It follows from flatness of ${}^{\bullet}\!\YM_{n}$ over $T$ that the $\kk$-dimension of ${}^{v}\!\YM_{n}$ is independent from $v \in \kk Q$. 
Thus we conclude that the desired formula holds true eve when $v$ is not sincere.  
\end{proof}

\subsection{The Dynkin case}

In this section $Q$ is a Dynkin quiver. 
To deal with the Dynkin case, we need to introduce the notion of regularity.
An element  $v \in \kk Q_{0}$ is called \emph{regular} if $ \sum_{i \in Q_{0}} v_{i}  \dim (e_{i} N) \neq 0 $
 for any indecomposable $\kk Q$-module $N$. 
 We note that a regular element is sincere. 
It is shown in \cite[Theorem 9.1]{QHA} that ${}^{v}\!\YM(Q)$ is finite dimensional if and only if $v$ is regular. 

\begin{theorem}[{\cite{Etingof-Rains} in the characteristic zero case}]\label{202303161409}
Let $Q$ be a Dynkin quiver. Assume that $v\in \kk Q_{0}$ is regular. 
Then we have 
\[
h_{{}^{v}\!\YM}(t) = \frac{1-t^{2h}}{(1-Ct+t^{2})(1-t^{2})}.
\]
\end{theorem}

\begin{proof}
By \cite[Theorem 9.9]{QHA}, we  have an isomorphism $\tuH({}^{v}\!\YYM) \cong {}^{v}\!\YM[u]$ of Adams graded algebras
 where the right hand side is the polynomial algebra in a single variable $u$ with cohomological degree $-2$ and Adams degree $2h$. 
Thus we have \[
h_{{}^{v}\!\YYM} = \sum_{n \geq 0} t^{2nh}h_{{}^{v}\!\YM}= (1-t^{2h})^{-1}h_{{}^{v}\!\YM}. 
\]
\end{proof}

\subsection{Proof of Theorem \ref{202303221917} }\label{202402211805} 
First note that by \cite[Definition 4.6]{BBK}, we have $\nu^{2} = \id_{\Pi}$. It follows that $P^{2} = 1$ and $(1-Pt^{h})(1+Pt^{h}) = 1-t^{2h}$.

Taking a field extension  of $\kk$ (if it is necessary), 
we may assume that there is a regular element $v \in \kk Q$ (see \cite[Example 5.11(2)]{QHA}). 
By \cite[Corollary 9.12, Remark 9.13]{QHA}, there exists an exact sequence of the following form 
\[
0 \to \tuD(\Pi)(-2h+2) \to {}^{v}\!\YM(-2) \xrightarrow{ \sfr_{\varrho}} {}^{v}\!\YM \to \Pi \to 0
\]
where $\sfr_{\varrho}$ denotes the multiplication map by $\varrho$. 
 On the other hand, we have $\tuD(\Pi) \cong {}_{\nu}\!\Pi(h-2)$ by \cite{BBK}. 
Thus, we have the equality 
$h_{\Pi}(t) -h_{{}^{v}\!\YM}(t) + t^{2}h_{{}^{v}\!\YM}(t) -t^{h}P h_{\Pi}(t) = 0$. 
Thus we obtain 
\[
h_{\Pi}(t) 
= \frac{1-t^{2}}{1-Pt^{h}}h_{{}^{v}\!\YM}(t) = 
\frac{1-t^{2h}}{(1-Pt^{h})(1-Ct+t^{2})} = 
\frac{1+Pt^{h}}{1-Ct+t^{2}}. 
\]
\qed

\appendix

\section{}

Let $B$ be a finite dimensional basic algebra and $\{e_{i}\}_{i=1}^{r}$ be a complete set of primitive orthogonal idempotent elements. 
We may regard $B$ as an algebra over $A:= \prod_{i=1}^{r} \kk e_{i}$. 
Let $\phi: B \to B$ be an algebra automorphism. We denote by $\overline{\phi}$ the induced permutation of $1,2,\dots, r$, i.e., 
for $i = 1,2,\dots, r$, $\overline{\phi}(i)$ denotes a unique number such that $B \phi(e_{i}) \cong B e_{\overline{\phi}(i)}$. 
We set $P:=(\delta_{\overline{\phi}(i)j})_{i,j=1}^{r}$. 
Finally let $J$ be the Jacobson radical of $B$ and $C:= [J/J^{2}]$  the dimension matrix of $B$-$B$-bimodule $J/J^{2}$.

\begin{lemma}\label{202402182021}
We have 
\[
PC=CP.
\]
\end{lemma}

\begin{proof}
For a $B$-$B$-bimodule $M$, it is straightforward to check that  $[{}_{\phi}M] =P_{\overline{\phi}}[M]$ and $[M_{\phi^{-1}}] = [M]P_{\overline{\phi}}$.  

It is clear that the automorphism $\phi^{-1}: B \to B$ preserves $J$ and $J^{2}$ and  hence induces an isomorphisms 
$J/J^{2} \to J/J^{2}, \ j   \ \mod J^{2} \mapsto \phi^{-1}(j) \  \mod J^{2}$ of $\kk$-vector spaces. 
This map gives an isomorphism ${}_{\phi}(J/J^{2}) \xrightarrow{\cong} (J/J^{2})_{\phi^{-1}}$ of $B$-$B$-bimodules. 
Combining the above observation, we obtain the desired conclusion.  
\end{proof}

\end{document}